\documentclass[a4paper]{article}
\usepackage[T1]{fontenc}
\usepackage[utf8]{inputenc}
\usepackage[pdftex]{graphicx}
\usepackage{amsmath}
\usepackage{amssymb}
\usepackage{bbm}
\usepackage{amsthm}
\usepackage[authoryear, round]{natbib}
\usepackage{mathtools}
\usepackage{float}
\usepackage{dsfont}
\usepackage{hyperref}
\hypersetup{colorlinks=true, linkcolor=[rgb]{0,0.5,0.6}, citecolor=[rgb]{0,0.5,0.6}, urlcolor=[rgb]{0,0.5,0.6}}

\theoremstyle{plain}
\newtheorem{theorem}{Theorem}[section]                                          
                          
\newtheorem{lemma}[theorem]{Lemma}

\theoremstyle{definition}
\newtheorem{definition}[theorem]{Definition}
\theoremstyle{remark}
\newtheorem{remark}[theorem]{Remark}
\newtheorem{example}[theorem]{Example}

\numberwithin{equation}{section}

\newcommand{\cupdot}{\mathbin{\mathaccent\cdot\cup}}

\begin{document}

\title{\textsc{Random walks on weighted, oriented percolation clusters}}
\author{Katja Miller\footnote{Fakult{\"a}t f{\"u}r Mathematik, Technische Universit{\"a}t M{\"u}nchen, Boltzmannstr. 3, 85748 Garching, Germany, katja.miller@tum.de, \url{http://www-m14.ma.tum.de/en/people/miller/}.}}
\date{\today}

\maketitle

\let\thefootnote\relax
\footnote{Research supported by Studienstiftung des deutschen Volkes.}

\vspace{0.7cm}

\noindent
\textsc{Abstract.} We consider a weighted random walk on the backbone of an oriented percolation cluster.  We determine necessary conditions on the weights for Brownian scaling limits under the annealed and the quenched law. This model is a random walk in dynamic random environment (RWDRE), where the environment is mixing, non-Markovian and not elliptic. We provide a generalization of results obtained previously by \cite{Birkner_et_al_2012}.

\vspace{0.5cm}

\section{Introduction}

Random walks in random environment (RWRE) are random walks whose transition kernels are not deterministic, but functions of some random field. This random field is called environment. We can interpret a RWRE as a two-stage experiment. At the first stage we determine the environment. Then, at the second stage, we determine the random walk with transition kernel depending on the environment. The law of both stages together is called \emph{annealed law}. If we keep the environment of the the first stage of the experiment fixed and only consider the experiment at the second stage,  we get the \emph{quenched} law of the RWRE. For an introduction to RWREs we refer to  lecture notes of \cite{Zeitouni_2004} for his course in Saint-Flour .

The environment can either be fixed during the evolution of the random walk, or it can be a stochastic process itself such that the transition kernels of the random walk change over time. The second case is called a random walk in dynamic random environment (RWDRE). Dynamic environments that have been studied in the past are i.i.d.~in \cite{Boldrighini_Minlos_Pellegrinotti_2004, Rassoul-Agha_Seppalainen_2005, Joseph_Rassoul-Agha_2011},
  exhibit small fluctuations in \cite{Bandyopadhyay_Zeitouni_2006}, are finite state Markov chains in \cite{Dolgopyat2009} or have some ellipticity and mixing properties in \cite{Andres_2014}. General ergodic Markovian environments which satisfy a coupling condition have been studied by \cite{Redig_Voellering_2013}. In a model by \cite{Dolgopyat_Keller_Liverani_2008} the random walk depends only weakly on the environment and some authors consider interacting particle systems as environments, e.g.~\cite{Birkner_et_al_2012}, \cite{Avena_Hollander_Reding_2010}, \cite{Hollander_Santos_2014} and \cite{Hilario_et_al_2014}. The relevance of some of these models for this paper is discussed in Section \ref{ssec: related_material}.

This paper is a generalization of a work by \cite{Birkner_et_al_2012}. They consider a directed random walk on an oriented percolation cluster, which can be considered as a RWDRE with a non-elliptic, non-reversible, Markovian random environment. Their random environment is the time reversal of a discrete time contact process. They prove a law of large numbers (LLN), an annealed central limit theorem (aCLT) and a quenched central limit theorem (qCLT). Their model describes simple population dynamics with local competition. Each site in the percolation cluster is considered habitable and can be occupied by at most one particle, while all other sites are unhabitable. The random walk on the habitable sites represents the ancestral line of one particle. In this paper we extend their model by allowing each site to be occupied by more than one particle.
We choose a carrying capacity for each site which is represented by a random field $K$. If a site is habitable in our model, then the species will populate this site with the maximal number of individuals allowed by $K$. We choose $K$ mixing so that it can model large scale features of different habitats like weather, soil conditions, altitude or seasons. The percolation cluster represents features of the habitat that only apply to a single site, e.g.~presence of a predator or shortage of food due to the presence of another species at a site. Thus our model is not only able to represent varying population densities, but also changes in habitats on two different scales. However, correlations in the habitat can only be allowed in the larger scale represented by the carrying capacities $K$.

We give conditions on the random field $K$ such that the LLN, aCLT and qCLT from \cite{Birkner_et_al_2012} still hold. We choose $K$ stationary, mixing and independent of the percolation cluster. The mixing property of the random field $K$ makes the environment non-Markovian and the percolation cluster makes it non-reversible. The random walk in our environment is a weighted random walk with weights $K$. The weights can be chosen in such a way that the weighted walk has a non-zero drift vector (see Example 3.2), while the unweighted walk has always vanishing speed.
The behaviour of a RWRE on the full lattice $\mathbb Z^d$ - with and without drift - is already interesting and not understood in all generality. We exploit the structure of the percolation cluster to establish these results for our model. 

\subsection{The Model}

The paper by \cite{Birkner_et_al_2012} provides a very detailed description of their model. We keep explanations in this section rather short and refer to their paper for a thorough discussion. We work on the discrete space  $V:=\mathbb Z^d\times \mathbb Z$, which we will refer to as the \emph{full lattice}. The first $d\geq 1$ dimensions in $V$ are \emph{space dimensions} and the last dimension is the \emph{time dimension}. We turn the lattice $V$ into an oriented graph $(V,E)$ with vertices $V$ by adding edges
\begin{align*}
	E:=\{|(x,n),(y,k)\rangle:(y,k)\in U^+(x,n))\},
\end{align*}
where $|(x,n),(y,k)\rangle$ denotes an oriented edge from $(x,n)$ to $(y,k)$ and
\begin{align}\label{eq:consecutive_sites}
	U^+(x,n):=\{(y,k)\in V: ||x-y||_\infty =1, k=n+1\}
\end{align}
is the set of consecutive vertices of $(x,n)$. The specific choice of the set of consecutive vertices $U^+$ is not important as long as it is finite and symmetric. 

Let $(\omega(x,n))_{(x,n)\in V}$ be a family of independent and identically distributed Bernoulli random variables with parameter $p\in(p_c,1]$ that represents a supercritical site percolation on the vertex set $V$. The constant $0<p_c<1$ is the critical probability of oriented site percolation on $(V,E)$. Existence and non-triviality of $p_c$ was proven in \cite{Grimmett_and_Hiemer_2002}. We say a site $(x,n)\in V$ is \emph{open}, if $\omega(x,n)=1$. Otherwise, we call it \emph{closed}. With the notion of open sites we can define open paths. A directed path on the oriented graph $(V,E)$ from vertex $(x,n)$ to vertex $(y,m)$ is called open, if all vertices on that path are open. For an open directed path from $(x,n)$ to $(y,m)$ we write $(x,n)\rightarrow(y,m)$. Analogously, we write $(x,n)\rightarrow \infty$ if there is an infinite, directed open path on $(V,E)$ starting in $(x,n)$. The percolation process $\xi^P:=(\xi^P_n)_{n\in\mathbb Z}$ is defined by
\begin{align*}
	\xi^P_n (x):=	\begin{cases} 	
				1 &\quad \text{if } (x,n)\rightarrow \infty \\
				0 &\quad \text{otherwise.}
			\end{cases}
\end{align*}
The backbone of the oriented percolation cluster is denoted by 
\begin{align*}
	\mathcal C:=\{(x,n)\in V: (x,n)\rightarrow \infty\}
\end{align*}
and is a proper subset of the oriented percolation cluster. It describes all sites that lie on an infinite directed open path on $(V,E)$.
On top of the percolation cluster we define weights $(K(x,n))_{(x,n)\in V}$ as a family of stationary $\mathbb R_{>0}$-valued random variables independent of $\omega$. It is important that the weights are strictly positive. We furthermore require the weights to be mixing. We now give the definition of the relevant mixing conditions used in this paper. For a brief overview on mixing conditions we refer the reader to the survey paper of \cite{Bradley_2005}.

\begin{definition}[Mixing conditions and mixing coefficients]\label{def_mixing}
Let $K$ be a random field on $V$. Denote by
\begin{align*}
	\sigma(K):=\sigma\left \{K(v): v\in V \right \}
\end{align*}
the $\sigma$-algebra of the weights and by 
\begin{align*}
	\mathrm{supp}(A):=\bigcap\left \{U\subset V: A \in \sigma(K(v): v\in U)\right \}
\end{align*}
the support of an event $A\in\sigma (K)$. Furthermore, we say that a set $C$ is a cone with apex in $(x,l) \in V$, if 
\begin{align*}
	C = \left \{ (y,k)\in V : k\geq l \text{ and } ||x-y||_\infty \leq |k-l|\right \}.
\end{align*}
This defines a cone with aperture $ \pi/2$.
\begin{enumerate}
\item We say that $K$ is \emph{$\alpha$-mixing (or strongly mixing) in space} w.r.t.~the law $\mathbb P$  if the mixing coefficients $(\alpha_n)_{n\in \mathbb N}$ satisfy $\alpha_n\xrightarrow{n\rightarrow\infty} 0$, where
\begin{align}\label{eq:mixing_property_alpha}
	\alpha_n := \sup &\left \{\left |\mathbb P(A\cap B)-\mathbb P(A)\mathbb P(B)\right |: \right . \\ \nonumber
		&\left . A,B\in\sigma(K), \mathrm{dist}^{\mathrm {s}}(\mathrm{supp}(A),\mathrm{supp}(B))>n\right \}
\end{align}
 and we take the distance in the first $d$ coordinates (space coordinates), i.e.~
\begin{align}
	\mathrm{dist}^{\mathrm {s}}(U,W)=\inf \left \{||x-y||_\infty:(x,n)\in U, (y,m)\in W \right \}.
\end{align}

\item We say that $K$ is \emph{$\phi$-mixing (or uniformly mixing) in time} w.r.t.~the law $\mathbb P$  if the mixing coefficients $(\phi_n)_{n\in \mathbb N}$ satisfy $\phi_n\xrightarrow{n\rightarrow\infty} 0$, where
\begin{align}\label{eq:mixing_property_phi}
	\phi_n := \sup &\left \{\left |\mathbb P(B|A)-\mathbb P(B)\right |: \right . \\ \nonumber
		&\quad \left . A,B\in\sigma(K), \mathbb P(A)>0, \mathrm{supp}(B)\subseteq C, C \text{ is a cone and } \right . \\ \nonumber 
		&\quad \left . \mathrm{dist}^{\mathrm {t}}(\mathrm{supp}(A),C)>n\right \}
\end{align}
and we take the distance in the  last coordinate  of $V$ (time coordinate),
\begin{align}
	\mathrm{dist}^{\mathrm {t}}(U,W)=\inf \left \{|n-m|:(x,n)\in U, (y,m)\in W \right \}.
\end{align}
\end{enumerate}
\end{definition}

We choose $K$ stationary, mixing and independent of $\omega$ for all our results. With this definition of the weights, the environment $\xi^K:=(\xi^K_n)_{n\in\mathbb Z}$ is given by the process
\begin{align}
	\xi^K_n (x):=	\begin{cases} 	
				K(x,n) &\quad \text{if } (x,n)\rightarrow \infty \\
				0 &\quad \text{otherwise.}
			\end{cases}
\end{align}
Since the weights $K$ are chosen to be strictly positive, the environment $\xi^K$ has zeros exactly at those sites where the percolation process $\xi^P$ has zeros and the percolation cluster is unchanged by the weights. Also, the percolation process $\xi^P$ is a Markov chain, while the environment process $\xi^K$ is not.

The random walk $(X_n)_{n\in\mathbb N}$ is defined on the environment $\xi^K$ as in the paper by \cite{Birkner_et_al_2012} for the case of i.i.d.~weights $K$, although in our paper the weights are not independent. We set $X_0=0$ and choose the transition kernel
\begin{align}\label{eq:trans_probabilities}
	{\mathbb P}(X_{n+1}=y|X_n=x, \omega, K)\propto K(y,n+1) \mathds 1_{\{(y,n+1)\in U^+(x,n)\cap\mathcal C\}}.
\end{align}
Set $\Omega:=\{0,1\}^{\mathbb Z^{d+1}}\times \mathbb R_{>0}^{\mathbb Z^{d+1}}\times \mathbb Z^{d}\times \mathbb N$ to be the sample space and equip it with the $\sigma$-algebra $\mathcal F:=\sigma (\omega(x,n), K(x,n), X_n:(x,n)\in V)$. We define two probability measures on this measurable space. 

\noindent
\fbox{
\begin{minipage}[t]{\textwidth - 4\fboxsep}
We denote by $\mathbb P$ the \emph{joint measure} of environment $\xi^K$ and the random walk $(X_n)$, which is the \emph{annealed} (or averaged) law. We write $\mathbb E$ for the expectation under the annealed law $\mathbb P$. Whenever we condition on the event $B_0:=\{(0,0)\in\mathcal C\}$, we denote this conditional law by a tilde, i.e.~$\tilde{\mathbb P}(\cdot):=\mathbb P(\cdot|B_0)$ and $\tilde{\mathbb E}(\cdot):=\mathbb E(\cdot|B_0)$.

We denote by $P_\xi$ the \emph{quenched} (or path-wise) law of the random walk, which is  $P_\xi(\cdot):=\mathbb P(\cdot|\xi^K)$. The expectation under the quenched law $P_\xi$ is denoted by $E_\xi$.
\end{minipage}
}

Throughout the paper, we choose to work with the supremum norm. The specific choice of a norm is not important for the results.  For any function $f:\mathbb N\rightarrow \mathbb R_+$ we write $\alpha_n\in\mathcal O(f(n))$ iff $\limsup \alpha_n/f(n)<\infty$ as $n\rightarrow \infty$. Finally, note that the constants $0<c,C<\infty$ are used in a generic sense and may take different values within the same set of equations.

\subsection{Results}

\begin{lemma}[LLN for polynomially time-mixing weights]\label{lem:lln}
Let $d\geq 1$ and $p\in(p_c, 1]$. If $K$ is independent of $\omega$, strictly positive, stationary and $\phi$-mixing in the time coordinate with mixing coefficients $\phi_n\in\mathcal O(n^{-(1+\delta)})$ for any $\delta>0$, then a LLN holds, i.e.~there is a constant $\vec \mu\in\mathbb R^d$ such that $||\vec \mu||_\infty<1$ and
\begin{align}
	P_{\xi}\left ( \frac{X_n}{n}\xrightarrow {n\rightarrow \infty} \vec\mu\right )=1 \quad \text{for } \tilde{ \mathbb P}\text{-a.e. }\xi^K.
 \end{align}
\end{lemma}

For the proof of the LLN we use a regeneration structure and we can express the drift vector $\vec \mu$ using regeneration times, see Equation \eqref{eq: drift_vector}.

\begin{theorem}[Annealed CLT for polynomially time-mixing weights]\label{thm:aCLT}
Let $d\geq 1$ and $p\in(p_c, 1)$. If $K$ is independent of $\omega$, strictly positive, stationary and $\phi$-mixing in the time coordinate with mixing coefficients $\phi_n\in\mathcal O(n^{-(2+\delta)})$ for some  $\delta>0$, then an aCLT holds, i.e.~for all continuous and bounded functions $f\in C_b(\mathbb R^d)$ 
\begin{align}
	\tilde{ \mathbb E}\left [ f\left ( \frac{(X_n-n\vec\mu)}{\sqrt n} \right )\right ]\xrightarrow{n\rightarrow\infty}\Phi(f),
\end{align}
where $\vec \mu$ is the same drift vector as in Lemma \ref{lem:lln}, $\Phi(f):=\int f(x)\Phi(\mathrm{d}x)$ and $\Phi$ is a non-trivial centred $d$-dimensional Gaussian law with full rank covariance matrix $\Sigma$.
\end{theorem}

While the LLN, Lemma \ref{lem:lln}, holds for $p=1$, we can prove the central limit theorems under the given mixing conditions only for $p<1$. For the aCLT on the full lattice, $p=1$, the main difficulty is to show non-degeneracy of the limit.

\begin{theorem}[Quenched CLT for exponentially space-time-mixing weights]\label{thm:qCLT}
Let $d\geq 2$ and $p\in(p_c, 1)$. If $K$ is independent of $\omega$, strictly positive,  stationary, $\phi$-mixing in the time coordinate with mixing coefficients $\phi_n\in\mathcal O(e^{-c_1n})$ and $\alpha$-mixing in space with mixing coefficients $\alpha_n\in\mathcal O(e^{-c_2n})$, $0<c_1,c_2<\infty$, then a quenched CLT holds with the same limit as in Theorem \ref{thm:aCLT}, i.e.~for all continuous and bounded functions $f\in C_b(\mathbb R^d)$ 
\begin{align}
	E_\xi\left [ f\left ( \frac{(X_n-n\vec\mu)}{\sqrt n} \right )\right ]\xrightarrow{n\rightarrow\infty}\Phi(f)\quad \text{for } \tilde{ \mathbb P}\text{-a.e. }\xi^K,
\end{align}
where $\vec \mu$ is the same drift vector as in Lemma \ref{lem:lln} and $\Phi$ is the same law as in Theorem \ref{thm:aCLT}.
\end{theorem}

Again, the qCLT holds with full rank covariance matrix $\Sigma$ for $p=1$ only under some additional assumptions, see Section \ref{ssec: related_material}.

\begin{remark} Since the publication of the paper, I have become aware that the definition of $\phi$-mixing for random fields as it is in Definition \ref{def_mixing} is in fact equivalent to finite dependence as explained in \cite{Bradley_1989}. However, we use the definition only in a context, where one of the two sets is cone-shaped. Thus, the theorems hold in fact, if the weights $K$ are uniformly cone-mixing in time.
\end{remark}

\subsection{Related Material}\label{ssec: related_material}
There are many closely related works that cover models similar to ours. A list of papers together with a brief description can be found in \cite{Birkner_et_al_2012}. Our environment $\xi^K$ is neither elliptic, nor reversible, Markovian or stationary with respect to $\tilde{\mathbb P}$. However, the environment is mixing and the environment seen from the particle is asymptotically stationary for constant weights $K\equiv 1$, which was shown by \cite{Steiber_15}. The first observation is used in the proofs in this paper, the second could result in alternative proofs of the aCLT for our model using standard methods, see e.g.~\cite{Zeitouni_2004}. There is also a second generalization of the underlying model by the authors of  \cite{Birkner_et_al_2012} themselves. They consider an environment, which is the time reversal of a Markov process generated by oriented percolation \cite{Birkner_et_al_2015}.

Our model falls also into the class of dynamic random conductance models, which is the classical set-up for RWDRE. On the full lattice, $p=1$, the environment is reversible and we can express the conductances in terms of our weights $K$. For any fixed time $n\in\mathbb N$, we define conductances in space between two neighbouring sites $x,y\in\mathbb Z^d$ by $c(x,y)= K(x,n)K(y,n)$. We get conductances for every time-slab that change dynamically in the time-coordinate. If we choose $K$ i.i.d.~we get a 2-dependent random conductance model.

We are interested in comparing results for weighted random walks on the percolation cluster, $p<1$, with weighted random walks on the full lattice, $p=1$. This provides us with a better understanding of the role of the backbone $\mathcal C$. If the weights $K$ are stationary and $\phi$-mixing with $\phi_n\in\mathcal O(n^{-(2+\delta)})$ for some $\delta>0$, then on the full lattice the aCLT holds directly by applying a central limit theorem for stationary, mixing sequences, e.g.~Theorem 18.5.2 in \cite{Ibragimov_Linnik_1971}. However, the limit law can be degenerate under these assumptions. Non-degenerate aCLTs were proven under more restrictive assumptions. For example \cite{Dolgopyat2009} treated the case where the conductances are i.i.d.~in the space dimensions and a finite state Markov chain in time. A more recent paper in this context is from \cite{Andres_2014}. He admits space-time mixing, non-Markovian environments, but requires bounded conductances.

We are also interested in counterexamples to show that our mixing condition in Theorem \ref{thm:aCLT} is sharp. We want to find $K$ such that the mixing condition of Theorem \ref{thm:aCLT} does not hold, i.e.~$\phi_n\notin\mathcal O(n^{-(2+\delta)})$ for any $\delta>0$, and there is no non-degenerate aCLT with Brownian scaling.  The following two examples apply to $p=1$ only and exhibit very strong traps for the random walk. Ideally we want examples that hold in the percolation case $p<1$, which is far more difficult.

The first example is from a paper by \cite{Berger_Salvi_2013}, who build on a construction by \cite{bramson2006} to show that it is possible to construct unbounded and mixing static random conductances such that a LLN does not hold. This construction can be applied to our dynamic model as well. Their conductances are polynomially mixing of order one in space and time, so the mixing coefficients are in $\mathcal O(n^{-1})$, but not in $\mathcal O(n^{-(1+\delta)})$ for any $\delta>0$. This example suggests that the condition on the mixing coefficients in Theorem \ref{thm:aCLT} is sharp, but for a proof we would need to make it work for $p<1$ as well.
Another interesting dynamic conductance model is described by  \cite{Buckley_2013}. He models the weights on the edges as independent, infinite-state Markov chains. The model is mixing in time and for large $n$ the mixing coefficients can be bounded below by $1/n$. As in the first example the mixing is slower than $\mathcal O(n^{-(1+\delta)})$ for any $\delta>0$.

These examples use weights to build traps for the random walker and force it into irregular behaviour for a long enough amount of time. The additional percolation cluster in our model helps the random walker to exit traps early. Since the percolation cluster is independent of the weights and has the ability to force the walker along the cluster it can create exit paths from traps formed by the weights. Consequently it is not possible to adapt the previous examples for our model. This also explains why we can prove non-degeneracy on the percolation cluster easier than on the full lattice.

\section{Mixing Properties of the Environment}

The key ingredient of our proofs is the mixing property of the percolation structure and environment. We will use it to show that we can define a regeneration structure, which is mixing itself, such that standard results for stationary, mixing sequences of random variables apply.

\begin{lemma}[The environment is mixing]\label{lem:environment_is_mixing}
Let $d\geq 1$ and $K$ be stationary and independent of $\omega$.
\begin{enumerate}
\item The processes $\xi^P$ and $\xi^K$ are stationary under the law $\mathbb P$.
\item The processes $\xi^P$ is mixing in space-time under the law $\mathbb P$ in the following sense: Fix $n\in \mathbb N$.  Let  $V_B\subset V$ be any cone shaped subset of $V$, i.e.~there is a site $(x, l)\in V$ and angle $\beta \in [\pi/4, \pi/2]$ such that
\begin{align}
	V_B:=\{(y,k)\in V: k\geq l \text{ and } ||x-y||_\infty\leq |k-l| \tan (\beta) \}.
\end{align} 
Let $V_A\subset V$ be such that $L:=|V_A|<\infty$ and $\mathrm{dist}(V_A, V_B)\geq n$. Then there exist constants $0<c,C<\infty$ such that for any two events $A,B\in \sigma(\xi^P_k(y):(y,k)\in V)$ with $\mathrm{supp}(A) \subseteq V_A$ and $\mathrm{supp}(B) \subseteq V_B$ we have 
\begin{align}\label{eq: mixing_condition_for_the_environment}
	\alpha_n^P(A,B):=\left | \mathbb P(A\cap B)-\mathbb P(A) \mathbb P(B) \right | \leq C  2^L L^2 e^{-cn}.
\end{align}
\end{enumerate}
\end{lemma}

\begin{remark} The environment is not stationary under the conditional law $\tilde{\mathbb P}$. However, the environment is mixing under  the conditional law $\tilde{\mathbb P}$, since the event $B_0$ can either be included in the event $A$ or in the event $B$ in Lemma \ref{lem:environment_is_mixing}. Therefore Equation \eqref{eq: mixing_condition_for_the_environment} holds with constants $L'=L+1$ and $c'=c/2$ also for $\tilde{\mathbb P}$.
\end{remark}

\begin{proof}[Proof of Lemma \ref{lem:environment_is_mixing}]
\emph{(i)} The process $\xi^P$ is stationary with respect to $\mathbb P$, which follows from the fact that the time-reversed process is a stationary discrete time contact process as explained in \cite{Birkner_et_al_2012}. The environment $\xi^K$ is stationary with respect to $\mathbb P$, since it is the product of two independent stationary processes. 

\emph{(ii)} First, define the length of the longest path on the oriented percolation cluster given by $\xi^P$  and starting in some point $(y,k)\in V$ by
\begin{align}\label{eq: length_of_longest_path}
	l(y,k):=\sup\{n\geq 1: \exists (y',k+n)\in V: (y,k)\rightarrow (y',k+n)\}.
\end{align}
Note that $l(y,k)=\infty$ if $(y,k)\in\mathcal C$. Define a subset $V_B\subset V$ and event $B$ as in Lemma \ref{lem:environment_is_mixing}. We will successively consider more complicated events for $A$. To begin with, let the second event be $A_1:=\{\xi^P_{k_1}(x_1)=0\}$ for some $(x_1,k_1)\in V$ such that $\mathrm{dist}(\{(x_1,k_1)\}, V_B)\geq n$. By Lemma A.1 in \cite{Birkner_et_al_2012} we know that
\begin{align}\label{eq: finite_open_path_length}
	\mathbb P\left (A_1 \cap \{l(x_1,k_1)\geq n\} \right )\leq Ce^{-cn}.
\end{align}
The event $\{l(x_1,k_1)< n\}\cap A_1$ is measurable with respect to $\sigma(\omega(v): v\in V\setminus V_B)$ and therefore independent of $B$. We can write
\begin{align*}
	\mathbb P(A_1\cap B)&=\mathbb P(A_1\cap B\cap \{l(x_1,k_1)<n\})+\mathbb P(A_1\cap B\cap \{l(x_1,k_1)\geq n\}) \\
	&\leq  \mathbb P(B)\mathbb P(A_1\cap \{l(x_1,k_1)<n\}) + \mathbb P(A_1\cap \{l(x_1,k_1)\geq n\}) \\
	&\leq  \mathbb P(B)\mathbb P(A_1) + \mathbb P(A_1\cap \{l(x_1,k_1)\geq n\}) 
\end{align*}
and similarly
\begin{align*}
	\mathbb P(A_1\cap B)&\geq\mathbb P(B)\mathbb P(A_1\cap \{l(x_1,k_1)<n\}) \\
	&=  \mathbb P(B)\left (\mathbb P(A_1) - \mathbb P(A_1\cap \{l(x_1,k_1)\geq n\}) \right )\\
	&\geq \mathbb P(B)\mathbb P(A_1) -\mathbb P(A_1\cap \{l(x_1,k_1)\geq n\}).
\end{align*}
We conclude, using Equation \eqref{eq: finite_open_path_length}, that 
\begin{align}
	\alpha_n^P(A_1,B) = \left | \mathbb P(A_1\cap B) -\mathbb P(A_1)\mathbb P(B)\right | \leq Ce^{-cn}.
\end{align}
The same upper bound follows for $\alpha^P_n(A_1^c,B) $ with $A_1^c:=\{\xi^P_{k_1}(x_1)=1\}$, if we use that
\begin{align*}
	\mathbb P(A_1^c)&=1-\mathbb P(A_1) \text{ and }\\
	\mathbb P(B\cap A_1^c)&=\mathbb P(B)-\mathbb P(B\cap A_1).
\end{align*}
We want to generalize this result to events that have support of more than one point. Consider events of the form 
\begin{align*}
	A_L^0:=\left \{\xi_{k_1}^P(x_1)=0\right \}\cap\ldots\cap \left \{\xi^P_{k_L}(x_L)=0\right \}
\end{align*}
for $L$ points $(x_1, k_1), \ldots, (x_L, k_L)\in V$ such that 
\begin{align*}
	\mathrm{dist}\left (\left \{(x_1, k_1), \ldots, (x_L, k_L) \right \}, V_B\right )\geq n
\end{align*}
By subadditivity, using the same steps as before, we get
\begin{align}\label{eq: mixing_for_all_zeros_event}
	\alpha_n^P(A_L^0, B)\leq CLe^{-cn}.
\end{align}
Observe that an arbitrary event of the form 
\begin{align*}
	A_L^s:=\left \{\xi^P_{k_1}(x_1)=s_1\right \}\cap\ldots\cap\left \{\xi^P_{k_L}(x_L)=s_L\right \}
\end{align*}
for any $s:=(s_1, \ldots, s_L)\in\{0,1\}^L$ can be written as the disjoint union of two events of the form $A_{L+1}$. For example $A_1^0=A_2^{(0,0)}\cupdot A_2^{(0,1)}$. Since we have already established the mixing property for events $A_1^0$ and $A_2^{(0,0)}$ in Equation \eqref{eq: mixing_for_all_zeros_event}, we can use the triangle inequality to get the mixing property for $A_2^{(0,1)}$. The same argument allows us to derive the bounds for arbitrary sets $A_L^s$, where we have to pay a price on the upper bound for each time we apply the triangle inequality. After adding all the upper bounds of the appearing terms, we get
\begin{align}
	\alpha_n^P(A_L, B)\leq CL^2e^{-cn}.
\end{align}
Finally, it remains to observe that any event $A$, with  $L=|V_A|$, can be written as a disjoint union of at most $2^L$ events of the type $A_l^s$, $1\leq l \leq L$ and the claim follows.
\end{proof}

\section{The Law of Large Numbers}

The process $\xi^K$ is not stationary with respect to $\tilde{\mathbb P}$, so we need to use a regeneration structure that has stationary increments.
The definition of the appropriate regeneration structure is similar to the case of i.i.d.~weights $K$ in \cite{Birkner_et_al_2012}. It uses additional random permutations to achieve a local construction of the random walk. For every $(x,n)\in V$ we let $\tilde \omega (x,n)$ be a random permutation of sites in $U^+(x,n)$, which is chosen from the set of all permutations according to the law
\begin{align}
	\mathbb P\left (\left . \tilde\omega(x,n)=(y_1, \ldots, y_N) \right | K \right )=\prod_{l=1}^{2^d} \frac{K(y_l,n+1)}{\sum_{k=l}^N K( y_k,n+1)}.
\end{align}
The sum runs over all consecutive vertices of $(x,n)$. The number of consecutive vertices $|U^+(x,n)|=2^d$ is the number of corners in a $d$-dimensional hypercube. Our construction of the local path will be measurable with respect to the $\sigma$-algebra of all weights and permutations in the time interval of interest, 
\begin{align}
	\mathcal G_n^m:=\sigma \left (\omega(y,k), \tilde \omega(y,k), y\in\mathbb Z^d, n\leq k <m \right ).
\end{align}
We need to know the length of the longest open path $l(x,n)$ starting at $(x,n)$. Then $l_k(x,n):=l(x,n)\wedge k$ is measurable with respect to $\mathcal G_n ^{n+k+1}$. For the local construction of the path we furthermore need the set of possible next steps if we want to stay on paths, which have at least length $k$. For any $k\geq -1$, we define this set as
\begin{align}
	M_k(x,n):=\begin{cases}  
				U^+(x,n) \quad & \text{if } k=-1, \\
				\left \{v\in U^+(x,n):l_k(v)=\max_{z\in U^+(x,n)} l_k(z)  \right \} \quad & \text{otherwise.}
			\end{cases}
\end{align}
Finally, we complete our auxiliary notation by choosing $m_k(x,n)\in M_k(x,n)$ to be the first element in the permutation $\tilde\omega(x,n)$. Given a percolation $\omega$, a permutation $\tilde\omega$ and a starting point $(x,n)\in V$ we finally define the local path $\gamma_k=\gamma_k^{(x,n)}$ by
\begin{align}
	\gamma_k(j):=	\begin{cases}
					(x,n)\quad & \text{if } j=0, \\
					m_{k-j-1}(\gamma_k(j-1)) \quad &\text{if } j=1,2,\ldots.
				\end{cases}
\end{align}
The law of the local path $(\gamma_\infty^{(x,n)}(j))_{j\geq 0}$ is the same as the law of the random walk $(X_j, n+j)_{j\geq 0}$ by Lemma 2.1 in \cite{Birkner_et_al_2012}. A more detailed description of this construction and a picture can be found in their paper as well.

For $p<1$ we want the set $S_{2m}$ to contain all sites $(x,n)\in V$ for which every directed open path returns to the space coordinate $x$ after $2m$ steps, 
\begin{align}
	S_{2m}:=\left\{ (x,n)\in V: (x,n)\rightarrow (x,n+2m),\ {\mathbb P}(X_{n+2m}=x|X_n=x)=1\right \}.	
\end{align} 
Note that we get a strictly positive lower bound on the probability that any site $(x,n)\in V$ is in this set conditioned that it is on the backbone $\mathcal C$ by considering a single path. Since it is already on the backbone we only need to make sure that all sites that are adjacent to the single path are closed, i.e.~if $p<1$
\begin{align}\label{eq: lower_bound_on_regeneration_path}
	\tilde{\mathbb P}\left (\left . v\in S_{2m}\right |v\in\mathcal C\right )\geq (1-p)^{2m(2d-1)}>0.
\end{align}
Our definition of the regeneration times differs from the paper of \cite{Birkner_et_al_2012} in the additional requirement that a regeneration can only happen at points in $S_{2m}$. Define the \emph{regeneration times} recursively by $T_0=0$ and 
\begin{align}\label{eq:regeneration_times}
	T_n:=\inf\left \{ k\geq T_{n-1}+2m:\gamma_{k-2m}(k-2m)\in\mathcal C\cap S_{2m} \right \}.
\end{align}
The corresponding regeneration increments are
\begin{align*}
	\tau_n:=T_n-T_{n-1}\quad \text{and}\quad Y_n:=X_{T_n}-X_{T_{n-1}}.
\end{align*}
The regeneration times are those times at which the local construction discovers a point that is in the backbone and is followed by an episode in the percolation cluster that forces the random walk to return after $2m$ steps independent of the weights $K$. Since behaviour of the random walk during these episodes does not depend on the weights $K$ it can be used to decrease dependency between regeneration increments by increasing $m$. Later in the proof we will choose $m$ large, see Equation \eqref{eq:m_arbitrary_makes_cc_vanish}, to show that the covariance matrix has full rank.

The regeneration times are not measurable with respect to the past of the environment $(\mathcal G_0^n)_n$. The local construction allows us to define \emph{potential regeneration times} $(\sigma_k)_{k\geq 0}$  for the $(i+1)$th regeneration by $\sigma_0=T_i$ and 
\begin{align}
\sigma_{k+1}=\sigma_k+l(\gamma_{\sigma_k}(\sigma_k))+2.
\end{align}
The potential regeneration times are those times at which the local construction discovers that a local path was finite and jumps to another branch, see Figure \ref{fig_regeneration_structure}. They  are $(\mathcal G_0^n)_n$-measurable and therefore stopping times. We only need to check at potential regeneration times whether all conditions for a regeneration are met. With this procedure we achieve minimal dependence on the future.

\begin{figure}[h]
\centering
\includegraphics[width=\textwidth]{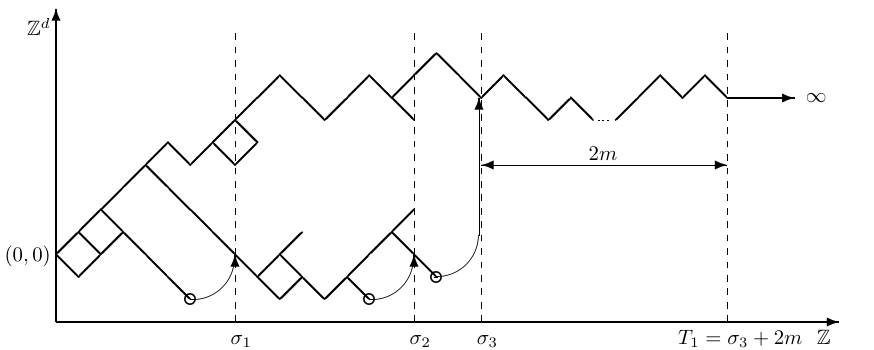}
\caption{Example for the regeneration structure in dimension $d=1$. The vertex set $V$ is not shown. The visible edges are those that can be reached from the origin $(0,0)$ by visiting open sites only. These edges are in the oriented percolation cluster of the origin. For a better visualization the permutations $\tilde\omega$ are chosen non-randomly and such that sites with smaller space coordinates are visited first. The local construction discovers three finite branches of the cluster before finding a regeneration time $T_1$. The end of each of these branches is marked by a circle. Afterwards the local discovery of the cluster is continued at the sites marked by the thin arrows. In this example only the topmost branch is connected by an open path to infinity and thus is in the backbone $\mathcal C$.}
\label{fig_regeneration_structure}
\end{figure}

\begin{lemma}[Increments of the random walk are ergodic]\label{lem:increments_are_ergodic}
Let $d\geq 1$, $K$ be independent of $\omega$, stationary and $\phi$-mixing in the time coordinate with mixing coefficients $(\phi_n)_{n\in\mathbb N}$. Then the process $(Y_n, \tau_n)_{n\in\mathbb N}$ is stationary and $\phi$-mixing with respect to $\tilde{\mathbb P}$ with mixing coefficients
\begin{align} 
	(\phi_n^X)_{n\in\mathbb N}=(\phi_{2mn} + 2\alpha^P_{2mn})_{n\in\mathbb N},
\end{align}
where $\alpha^P_n=Ce^{-cn}$, $n\in\mathbb N$ are the mixing coefficients for $\xi^P$ from Lemma \ref{lem:environment_is_mixing}, Equation \eqref{eq: mixing_condition_for_the_environment}.
\end{lemma}

\begin{proof}
Fix a site $(x,l)\in V$ such that $||x||_\infty \leq l$. Then $\tilde{\mathbb P}(\gamma_{T_n}(T_n)=(x,l))>0$. We observe that for all $n\in\mathbb N$ by the local construction of the random walk there exists an event 
\begin{align*}
	A'\in\sigma\left (\omega(y,k), \tilde\omega(y,k): (y,k)\in V, 0\leq k < T_n\right )
\end{align*}
such that 
\begin{align}\label{eq:split_regeneration_time}
	\{\gamma_{T_n}(T_n)=(x,l)\}=A'\cap \{(x,l)\rightarrow \infty\}\subset B_0.
\end{align}
Let $\theta_z:\Omega\mapsto \Omega$, $z\in V$ be the standard shift operator such that $(\theta_{z}\omega)(z')=\omega(z+z')$ for any $\omega \in \Omega$, $z,z'\in V$. Then we can write
\begin{align}\label{eq:split_shifted_regeneration_time}
	\theta^{-1}_{(x,l)}(\{\gamma_{T_n}(T_n)=(x,l)\})=\theta^{-1}_{(x,l)}(A')\cap B_0.
\end{align}
Thus, for every event $A\in\sigma(\xi^K_k(y):(y,k)\in V)$ we have
\begin{align*}
\mathbb P(\theta_{\gamma_{T_n}(T_n)}(A)\cap B_0)&= \mathbb E\left [ \mathbb P \left (\left . \theta_{(x,l)}(A)\cap B_0 \right | \gamma_{T_n}(T_n)=(x,l)\right )\right ]\\
&\stackrel{\mathclap{\eqref{eq:split_regeneration_time}}}= \ \ \mathbb E\left [ \mathbb P \left (\left . \theta_{(x,l)}(A)\right | \gamma_{T_n}(T_n)=(x,l)\right )\right ]\\
&= \mathbb E\left [ \mathbb P \left ( A \left  | \theta^{-1}_{(x,l)} (\{\gamma_{T_n}(T_n)=(x,l)\})\right .\right )\right ]\\
&\stackrel{\mathclap{\eqref{eq:split_shifted_regeneration_time}}}= \ \ \mathbb E\left [ \mathbb P \left ( A\cap B_0 \left  | \theta^{-1}_{(x,l)} (\{\gamma_{T_n}(T_n)=(x,l)\})\right .\right )\right ]\\
&=\mathbb P (A\cap B_0).
\end{align*}
Consequently $\tilde{\mathbb P}(\theta_{\gamma_{T_n}(T_n)}(A))=\tilde{\mathbb P} (A)$ and both processes are stationary with respect to $\tilde{\mathbb P}$.

Denote by $\mathcal W$ the $\sigma$-algebra that contains all possible paths of the random walk, namely
\begin{align*}
	\mathcal W_k^l :=\sigma \left (\{(X_i(\omega), i)\}_{i=k}^l: \omega\in\Omega  \right )
\end{align*}
and $\mathcal W=\mathcal W_0^\infty$. Then the mixing coefficients for the process $(X_{T_n}-X_{T_{n-1}})_{n\in\mathbb N}$ are given by
\begin{align}
	\phi^X_n=\sup_{N\in\mathbb N} \sup\limits_{\substack{W\in\mathcal W, \\ A^N:=W\cap \mathcal W_0^{T_N}, \\ B^N:=W\cap \mathcal W_{T_{N+n}}^\infty}}  \left | \frac{\tilde {\mathbb P}(A^N\cap B^N)}{\tilde {\mathbb P}(A^N)}-\tilde {\mathbb P}(B^N)\right |.
\end{align}
Note that by definition of the random walk and since $K>0$ we have $\tilde {\mathbb P}(A^N)>0$ for all $A^N\in W\cap\mathcal W_0^{T_N}$.
We will from now on leave out the subscripts of the suprema. Furthermore, note that for every $(x,l)\in V$ exists an event
\begin{align*}
	A^N_{(x,l)}\in\sigma(\omega(y,k),  \tilde{\omega}(y,k): y\in\mathbb Z^d, k<l)
\end{align*}	
such that
\begin{align*}
	A^N\cap\left \{(X_N, T_N)=(x,l)\right \}=A^N_{(x,l)}\cap\{\xi^P_l(x)=1\}.
\end{align*}
This allows us to split up the events into disjoint subsets depending on where the path ends. We rewrite the mixing coefficients as
\begin{align*}
	\phi^X_n&=\sup \sup \frac{1}{{\tilde {\mathbb P}(A^N)}}\left | \sum_{(x,l)\in\mathbb Z^{d+1}} \tilde {\mathbb P}\left (A^N\cap B^N\cap \{(X_N, T_N)=(x,l)\}\right ) \right . \\
	&\qquad \qquad \qquad \qquad \qquad \left . \vphantom{ \sum_{(x,l)}} -\tilde {\mathbb P}\left (A^N\cap \{(X_N, T_N)=(x,l)\}\right )\tilde {\mathbb P}\left (B^N\right )\right | \\
	&=\sup \sup \frac{1}{{\tilde {\mathbb P}(A^N)}}\left | \sum_{(x,l)} \tilde {\mathbb P}\left (A^N_{(x,l)}\cap B^N\cap \{\xi^P_l(x)=1\}\right ) \right . \\
	&\qquad \qquad \qquad \qquad \qquad \left . \vphantom{ \sum_{(x,l)}} -\tilde {\mathbb P}\left (A^N_{(x,l)}\cap \{\xi^P_l(x)=1\}\right )\tilde {\mathbb P}\left(B^N\right )\right |\\ 
	&=\sup \sup \frac{1}{\mathbb P(B_0){\tilde {\mathbb P}(A^N)}}\left | \sum_{(x,l)}  {\mathbb P}\left (A^N_{(x,l)}\cap B^N\cap \{\xi^P_l(x)=1\}\right ) \right . \\
	&\qquad \qquad \qquad \qquad \qquad \left . \vphantom{ \sum_{(x,l)}} -{\mathbb P}\left (A^N_{(x,l)}\cap \{\xi^P_l(x)=1\}\right )\tilde {\mathbb P}\left (B^N \right)\right |.
\end{align*}
The last equation follows from the fact that $B_0\subset A^N_{(x,l)}\cap\{\xi^P_l(x)=1\}$. We can use independence of $A^N_{(x,l)}$ and $\{\xi^P_l(x)=1\}$ and the mixing property of the weights $K$ to get
\begin{align}
	\phi^X_n &\leq \sup \sup \frac{1}{\mathbb P(B_0){\tilde {\mathbb P}(A^N)}}\left | \sum_{(x,l)}  \mathbb P\left (A^N_{(x,l)}\right )\mathbb P\left (B^N\cap \{\xi^P_l(x)=1\}\right ) \right . \\ \nonumber
	&\qquad \qquad \qquad \qquad \qquad \left . \vphantom{ \sum_{(x,l)}}-\mathbb P\left (A^N_{(x,l)}\right )\mathbb P\left (\{\xi^P_l(x)=1\}\right )\tilde {\mathbb P}\left (B^N\right )\right | + \mathcal E_1(n),
\end{align}
where 
\begin{align}\label{eq_error_term_1}
	\mathcal E_1(n):=\frac{1}{\mathbb P(B_0)}\frac{\sum_{(x,l)} \mathbb P\left (A^N_{(x,l)}\right )}{{\tilde {\mathbb P}(A^N)}} \phi_{T_{N+n}-T_N}\leq \frac{1}{\mathbb P(B_0)}\frac{\sum_{(x,l)}\mathbb P\left (A^N_{(x,l)}\right )}{\tilde {\mathbb P}(A^N)} \phi_{2mn}.
\end{align}
We use stationarity of the environment and $\mathbb P(B_0)>0$ to see that in fact
\begin{align}\label{eq: sum_in_mixing_upper_bound_is_finite}
	\tilde {\mathbb P}(A^N)=\sum_{(x,l)}\frac{{\mathbb P}\left (A^N_{(x,l)}\cap\{\xi^P_l(x)=1\}\right)}{\mathbb P(B_0)} = \sum_{(x,l)}{\mathbb P}\left (A^N_{(x,l)}\right ).
\end{align}
Note that since $K$ is $\phi$-mixing instead of $\alpha$-mixing, the factor $\sum_{(x,l)}\mathbb P(A_{(x,l)}^N)$ appears in the upper bound in Equation \eqref{eq_error_term_1}, which cancels with the denominator and makes the error term finite.

We can use the mixing property of $\xi^P$ from Lemma \ref{lem:environment_is_mixing} to factor $\mathbb P(B^N\cap \{\xi^P_l(x)=1\})$ and $\tilde{\mathbb P}(B^N)=\mathbb P(B^N\cap B_0)/\mathbb P(B_0)$. This leads to the upper bound
\begin{align}
	\phi^X_n&\leq \sup \sup \frac{1}{\mathbb P(B_0){\tilde {\mathbb P}(A^N)}}\left | \sum_{(x,l)}  \mathbb P\left (A^N_{(x,l)}\right )\mathbb P\left (B^N\right )\mathbb P\left (\xi^P_l(x)=1\right ) \right . \\ \nonumber
	&\qquad \qquad \qquad \qquad  \left . \vphantom{ \sum_{(x,l)}}-\mathbb P\left (A^N_{(x,l)}\right )\mathbb P\left (\xi^P_l(x)=1\right )\mathbb P\left (B^N\right )\right | + \mathcal E_1(n)+2\mathcal E_2(n) \\
	&= \sup \sup \left (\mathcal E_1(n)+2\mathcal E_2(n)\right ), 
\end{align}
with 
\begin{align}\label{eq_error_term_2}
	\mathcal E_2(n):=\frac{1}{\mathbb P(B_0)}\frac{\sum_{(x,l)} \mathbb P\left (A^N_{(x,l)}\right )}{\tilde {\mathbb P}(A^N)} \alpha^P_{T_{N+n}-T_N}\leq \frac{1}{\mathbb P(B_0)}\frac{\sum_{(x,l)}\mathbb P\left (A^N_{(x,l)}\right )}{\tilde {\mathbb P}(A^N)} \alpha^P_{2mn}.
\end{align}
Combining Equations \eqref{eq_error_term_1},  \eqref{eq: sum_in_mixing_upper_bound_is_finite} and \eqref{eq_error_term_2} tells us that overall the sequence of regeneration increments is $\phi$-mixing and the mixing coefficients are bounded above by 
\begin{align}
	\phi^X_n\leq \frac{1}{\mathbb P(B_0)}\left ( \phi_{2mn} + 2\alpha^P_{2mn}\right ).
\end{align}
\end{proof}

With this preparation the LLN, Lemma \ref{lem:lln}, follows directly from the previous results.

\begin{proof}[Proof of Lemma \ref{lem:lln}]
By Lemma \ref{lem:increments_are_ergodic}, the sequence $(Y_n)_{n\in\mathbb N}$ is stationary and mixing and therefore ergodic. The law of large numbers follows from the ergodic theorem (\citealp{Birkhoff_1931}) together with standard arguments from renewal theory and the drift vector takes the usual form,
 \begin{align}\label{eq: drift_vector}
 	\vec \mu =\frac{\tilde{\mathbb E}[X_{T_1}]}{\tilde{\mathbb E}[T_1]}.
 \end{align}
\end{proof}

The next example shows that the average $\vec\mu$ can indeed be non-zero on the full lattice, even if the weights $K$ are independent in time. The example was provided in private communication by Noam Berger.

\begin{example}
Let $d=1$. We construct an environment from bounded weights that are independent in time such that the random walk is ballistic in the space coordinate, i.e.~$\mu \neq 0$. Let $(\beta (n))_{n\in\mathbb N}$ be a family of independent random variables, each of them uniformly distributed on the set $\{0,1,2\}$. Choose weights for all $x\in\mathbb Z$ according to
\begin{align*}
	K(x,n)=((\beta(n)+3|x|+x)\mod 3)+1\in\{1,2,3\}.
\end{align*}
Then the average speed is
\begin{align*}
	\mu=\mathbb E\left [\frac{K(1,n)-K(-1,n)}{K(1,n)+K(-1,n)}\right ]=-1/90<0
\end{align*}
for any $n\in\mathbb N$.
\end{example}

\section{The Annealed Central Limit Theorem}

The aCLT follows without much additional work from the results we already established for the LLN.

\begin{proof}[Proof of Theorem \ref{thm:aCLT}]
We begin with the proof for $p<1$, where we have to consider the percolation cluster. As defined in Equation \eqref{eq: length_of_longest_path} the random variable $\tau^{(x,n)}$ denotes the length of the longest open path starting at the site $(x,n)$. The proof is similar to the proof of Lemma 2.5 in \cite{Birkner_et_al_2012}, since $K$ is independent of $\omega$ and the bounds derive from the structure of the open cluster. In particular, the increments $(\sigma_{k+1}-\sigma_k)\leq l(\gamma_{\sigma_k}(\sigma_k))$ are dominated by a random variable, which is independent of the weights $K$. Furthermore the number of trials to find a regeneration time is dominated by a geometric random variable with success probability $\mathbb P(B_0)(1-p)^{2m(2d-1)}>0$ by Equation \eqref{eq: lower_bound_on_regeneration_path}. Consequently, the first regeneration time has exponential tails,
\begin{align}\label{eq:increment_tail_bound}
	\tilde{\mathbb P}\left (T_1>n  \right )\leq Ce^{-cn}.
\end{align}
The same bound holds for the space increment $||Y_1||_\infty$, since $||Y_1||_\infty \leq T_1$ for all $n\in\mathbb N$. We have shown in Lemma \ref{lem:increments_are_ergodic} that the sequence of regeneration increments $(Y_n, \tau_n)_{n\in\mathbb N}$ is stationary and $\phi$-mixing with coefficients $\phi^X_{n}$ for $n$ large enough. Therefore, all increments have exponential tail bounds. Under the mixing conditions of Theorem  \ref{thm:aCLT}, $\phi_n \in\mathcal O(n^{-(2+\delta)})$ for some $\delta>0$, we get that
\begin{align}\label{eq:mixing_is_quick_enough}
	\sum_{k=1}^\infty (\phi_{k}^X)^{\frac{1}{2}}\leq\frac{1}{\mathbb P(B_0)^{1/2}}\sum_{k=1}^\infty (\phi_{2mk}+2\alpha_{2mk}^P)^{\frac{1}{2}}<\infty.
\end{align}
This is the condition of \cite{Ibragimov_Linnik_1971} for the CLT for $\phi$-mixing sequences, Theorem 18.5.2. We prove the aCLT first in the case $d=1$. Define centred random variables $Z_n=Y_n-\tilde{\mathbb E}[Y_n]= Y_n-\tilde{\mathbb E}[Y_1]$ for all $n\in\mathbb N$. 
By Equation \eqref{eq:increment_tail_bound} we know that  $\tilde{\mathbb E}\left [|Z_{n}|^{D+2}\right ]<\infty$ and $\tilde{\mathbb E}[\tau_n^{D+2}]< \infty$. Since $Z_n$ is centred, Theorem 17.2.3 in \cite{Ibragimov_Linnik_1971} and Equation \eqref{eq:mixing_is_quick_enough} imply that 
\begin{align}\label{eq:upper_bound_in_m}
	\left | \sum_{n=1}^\infty \tilde{\mathbb E}[Z_0Z_n]\right|&\leq 2 \sum_{n=1}^\infty \tilde {\mathbb E}[Z_0^2]^{1/2}\tilde {\mathbb E}[Z_n^2]^{1/2}(\phi_{n}^X)^{1/2}\nonumber\\
	&=2\tilde{\mathbb E}[Z_0^2]\sum_{n=1}^\infty (\phi_{n}^X)^{1/2}<\infty.
\end{align}
Furthermore, the sum can be made arbitrary small if we choose $m$ large enough. Choose $m$ so that $|\sum_{n=1}^\infty (\phi_{n}^X)^{1/2}|<1/4$.
Then the variance is strictly positive,
\begin{align}\label{eq:m_arbitrary_makes_cc_vanish}
	\sigma^2=\tilde {\mathbb E}[Z_0^2]+2\sum_{n=1}^\infty \tilde{\mathbb E}[Z_0 Z_n]\geq \tilde {\mathbb E}[Z_0^2]\left (1-4\sum_{n=1}^\infty (\phi_{n}^X)^{1/2}\right )>0.
\end{align}
This choice of the distance $2m$ between two pieces of the regeneration increments allows us to conclude that the variance is strictly positive and the central limit theorem has a non-degenerate limit. Here, we use the percolation cluster explicitly to bound the variance away from zero. Using a central limit theorem for stationary and $\phi$-mixing sequences, e.g. Theorem 18.5.2 in \cite{Ibragimov_Linnik_1971}, and renewal arguments (\citealp{Kuczek_1989}) we get a non-degenerate central limit theorem for the sequence $(Y_n, \tau_n)_{n\in\mathbb N}$. 

Furthermore, we can generalize this result to the multivariate case using e.g.~Lévy’s continuity theorem as in \cite{Rio_2013}, Corollary 4.1. In this case, we have to choose $m$ large enough, such that the covariance matrix $\Sigma$ has full rank. The covariance matrix $\Sigma:=(\Sigma_{ij})_{1\leq i,j\leq d}$ is given by
\begin{align}\label{eq:covariance_matrix_of_CLT}
	\Sigma_{ij}=\tilde{\mathbb E}\left [\langle Z_{0},e_i \rangle \right ]+2 \sum_{k=1}^\infty \tilde{\mathbb E}[\langle Z_{0}, e_i\rangle \langle Z_{k}, e_j \rangle],
\end{align}
where $\langle \cdot, \cdot \rangle$ denotes the usual Euclidean scalar product and $\{e_1, \ldots, e_d\}$ is the canonical basis of $\mathbb Z^d$.
\end{proof}

\section{The Quenched Central Limit Theorem}
The main idea for the proof is to study a pair of random walks on the same environment and show that their behaviour is close enough to the behaviour of two random walks on independent copies of the environment. As we did for the regeneration structure for a single random walk we define the sequence of regeneration times for two random walks starting at times $T_0=T'_0=0$ for $ j\geq 1$ by
\begin{align}
	\begin{split}
		&T_j:=\inf\left \{k>T_{j-1}+2m:\gamma^{(x,0)}_{k-2m}(k-2m)\in\mathcal C\cap S_{2m}\right\}, \\
		&T_j':=\inf\left \{k>T'_{j-1}+2m:\gamma'^{(x',0)}_{k-2m}(k-2m)\in\mathcal C\cap S_{2m}\right\}.
	\end{split}
\end{align}
Set $J_0=J'_0=0$ and for $m\in\mathbb N$ and define auxiliary times
\begin{align}
	\begin{split}
		J_j&:=\inf\{k>T_{j-1}:T_k=T'_{k'} \text{ for some } k'>J'_j\} \quad \text{and}\\
		J_j'&:=\inf\{k>T'_{j-1}:T'_{k'}=T_k \text{ for some } k>J_j\}.
	\end{split}
\end{align}
Define the sequence of simultaneous regeneration times by 
\begin{align}
	T_m^{\text{sim}}:=T_{J_m}=T'_{J'_m},\quad m\geq 0
\end{align}
or recursively $T_0^{\text{sim}}=0$ and
\begin{align}
	T_m^{\text{sim}}=\min\left( \{T_j:T_j>T_{m-1}^{\text{sim}}\}\cap\{T'_j:T'_j>T_{m-1}^{\text{sim}}\}\right ).
\end{align}
The increments $Y_k,Y'_k, \tau_k$ and $\tau'_k$ are defined as in the single walk case and we set for $m,l\in \mathbb N$
\begin{align}
	\begin{split}
		\tilde X_m:=X_{T_m}, &\quad \tilde X'_m:=X'_{T'_m} \\
		\hat X_l:=X_{T_l^{\text{sim}}}, &\quad\hat X'_l:=X'_{T_l^{\text{sim}}}.
	\end{split}
\end{align}
Finally denote the pieces between simultaneous regenerations $\Xi_m\in \mathcal W := \mathbb F\times\mathbb  F\times\mathbb Z^d \times\mathbb Z^d$ by
\begin{align}
	\Xi_m:=\left ( (Y_k, \tau_k)_{k=J_{m-1}+1}^{J_m}, (Y'_k, \tau'_k)_{k=J'_{m-1}+1}^{J'_m}, X_{T_{J_m}},X'_{T'_{J'_m}} \right ),
\end{align}
where $\mathbb F:=\bigcup_{n=1}^\infty (\mathbb Z^d \times \mathbb N)^n$.
We need some more notation to indicate when we are considering two random walks simultaneously on the same percolation cluster. Take two starting points for the random walks $x,x'\in\mathbb Z^d$. Let 
\begin{align*}
	B_{x,x'}:=\{\xi^P_0(x)=1\}\cap \{ \xi^P_0(x')=1\}
\end{align*} 
be the event that both starting points are in the backbone $\mathcal C$. Conditioned on $B_{x,x'}$ let $X:=(X_n)_n$ and $X':=(X'_n)_n$ be two independent random walks started at $(x,0)$ and $(x',0)$ respectively and both with transition probabilities  as in Equation \eqref{eq:trans_probabilities}. Write for the law of the two walkers conditioned on $B_{x,x'}$
\begin{align}
	\tilde{\mathbb P}^{\mathrm{joint}}_{x,x'} (\cdot)=\mathbb P^{\mathrm{joint}}_{x,x'} \left (\left . \cdot \ \right  | B_{x,x'}\right )=\mathbb P^{\mathrm{joint}}\left ( \left . \cdot \ \right | X_0=x, X'_0=x', B_{x,x'}\right ),
\end{align}
where the superscript indicates that the two walks run on the same realization of the environment. We will describe the joint law by comparing it to the law of two independent random walks $\tilde{\mathbb P}^{\mathrm{ind}}_{x,x'}$, which is the product measure of two independent copies of single random walks with laws $\tilde{\mathbb P}_{x}^1$ and $\tilde{\mathbb P}_{x'}^2$ on two independent copies $\Omega_1, \Omega_2$ of the environment,
\begin{align*}
	\tilde{\mathbb P}^{\mathrm{ind}}_{x,x'} (\cdot)=\tilde{\mathbb P}_{x}^1(\cdot) \tilde{\mathbb P}^2_{x'}(\cdot).
\end{align*}
To describe the second random walk, let $\tilde\omega'$ be another, independent random permutation distributed like $\tilde\omega$ and define paths $\gamma_k'^{(x,n)}$ analogously to $\gamma_k^{(x,n)}$ using $\tilde\omega'$ instead of $\tilde\omega$. For given $n,k$ the construction of both paths are measurable w.r.t 
\begin{align}
	\hat{\mathcal G}_n^{k}:=\sigma\left ( \omega(y,i), \tilde\omega(y,i), \tilde\omega'(y,i)\ : \ y\in\mathbb Z^d, n\leq i<k\right ).
\end{align}
Conditioned on $B_{x,x'}$ we may couple the random walks by
\[(X_k,k)=\lim_{n\rightarrow\infty}\gamma_n^{(x,0)}(k),\quad (X'_k,k)=\lim_{n\rightarrow\infty}\gamma_n'^{(x',0)}(k).\]

The proof of Theorem \ref{thm:qCLT} is analogous to the proof of Theorem 2 in the paper of \cite{Birkner_et_al_2012}. However, some of the lemmas along the way have to be modified. Most of the proofs in this paper are kept rather short, if a similar and more detailed version can be found in the original paper. Here, we will list the essential adaptations needed to make it suitable for our problem. We get the exponential bounds on the joint regeneration times with the same argument as in the proof of Theorem \ref{thm:aCLT}, Equation \eqref{eq:increment_tail_bound}, i.e.~for all $x,x'\in\mathbb Z^d$
\begin{align}\label{eq:simultaneous_increment_tail_bound}
	\tilde{\mathbb P}_{x,x'}^{\mathrm{joint}}\left ( T^{\mathrm{sim}}_1>n \right )\leq Ce^{-cn}.
\end{align}
Also, the sequence of joint regeneration increments is again stationary and $\alpha$-mixing by a similar reasoning as is used in Lemma \ref{lem:increments_are_ergodic}.

\begin{lemma}[Total variation distance of joint and independent law, cf.~Lemma 3.4 in \cite{Birkner_et_al_2012}]\label{lem:tv_distance}
	There exist constants $0<c,C<\infty$ such that for all $x,x'\in\mathbb Z^d$
\begin{align*}
	\left \|{\mathbb P}^{\mathrm{joint}}_{x,x'} (\Xi_1=\cdot)-{\mathbb P}^{\mathrm{ind}}_{x,x'} (\Xi_1=\cdot)\right \|_{\text{TV}}\leq Ce^{-c||x-x'||},
\end{align*}
where $||\cdot||_{TV}$ is the total variation norm.
\end{lemma}

\begin{proof}
As in the original paper, without loss of generality, we prove the lemma for two start points $x=0$ and $x'e_1$, where $e_1$ is the first coordinate vector in $\mathbb Z^d$ and $x'>2m$. Let $\Omega_1$, $\Omega_2$ and $\Omega_3$ be three independent copies of environment and permutations $\Omega_i$, i.e. for all $i\in\{1,2,3\}$ we define
\begin{align*}
	\Omega_i:=\{ \omega_i(v), K_i(v),  \tilde \omega_i(v):v\in V\}.
\end{align*}

Throughout this proof, we will add $\Omega_i$ as an argument to our random variables to indicate which realization of percolation and permutation is used in the construction. Detailed definitions can be found in the proof of Lemma 3.4 in \cite{Birkner_et_al_2012}. For example, we write
\begin{align*}
	B_{x,x'}(\Omega_i, \Omega_j):=\{\xi^P_0(x; \Omega_i)=\xi^P_0(x'; \Omega_j)=1\}
\end{align*}
for the condition to start the walks on the backbones of $\Omega_i$ and $\Omega_j$ respectively, 
\begin{align*}
	T_{i,j}^\mathrm{sim}&:=T^\mathrm{sim}(\Omega_i, \Omega_j)\\
	&:=\inf\left \{n\geq 1: \xi^P_n(\gamma_n^{(x,n)}(n;\Omega_i);\Omega_i)=\xi^P_n(\gamma_n^{(x,n)}(n;\Omega_j);\Omega_j) =1 \right \}
\end{align*}
for the simultaneous regeneration times and
\begin{align*}
	&\Xi_1(\Omega_i,\Omega_j)\\
	&:=\left ( (Y_k(\Omega_i), \tau_k(\Omega_i))_{k=1}^{J_1(\Omega_i, \Omega_j)}, (Y'_k(\Omega_j), \tau'_k(\Omega_j))_{k=1}^{J'_1(\Omega_i, \Omega_j)}, X_{T_{i,j}^\mathrm{sim}}(\Omega_i),X'_{T^{\prime\mathrm{sim}}_{i,j}}(\Omega_j) \right ).
\end{align*}
for the simultaneous regeneration increments. To construct a simultaneous regeneration increment of two independent walks $\Xi_{x,x'}^\text{ind}$, we will start one random walk at $x=0$ on $\Omega_1$ and another random walk at $x'$ on $\Omega_2$. Similarly we construct the simultaneous regeneration increment of two walks on the same cluster $\Xi_{x,x'}^\text{joint}$ by starting two random walks in $x$ and $x'$ respectively both on $\Omega_3$. It is convenient to write 
\begin{align*}
	\Xi_{x,x'}^\text{joint}&:=\begin{cases}
		\Xi_1(\Omega_3, \Omega_3), \quad &\text{if }B_{x,x'}(\Omega_3, \Omega_3)\text{ occurs,}\\
		\Delta \quad &\text{otherwise,}
	\end{cases}\\
	\Xi_{x,x'}^\text{ind}&:=\begin{cases}
		\Xi_1(\Omega_1, \Omega_2), \quad &\text{if }B_{x,x'}(\Omega_1, \Omega_2)\text{ occurs,}\\
		\Delta \quad &\text{otherwise,}
	\end{cases}
\end{align*}
with some cemetery state $\Delta$. If we start the random walks far enough apart, then with high probability the regeneration event will happen in two disjoint subsets of $V$ that have a distance of $x'/2$. This allows us to use the mixing properties of the environment, Lemma \ref{lem:environment_is_mixing}. Define the two disjoint subsets of $S_1,S_2\subseteq V$ by
\begin{align*}
	S_1&:=\{(y,k)\in V: |y|\leq k \text{ for all } 0\leq k \leq |x-x'|/4\},\\
	S_2&:=\{(y,k)\in V: |y-x'|\leq k \text{ for all } 0\leq k \leq |x-x'|/4\}.
\end{align*}
Then $\mathrm{dist}(S_1, S_2)=x'/2$ as shown in Figure \ref{fig_total_variation_distance}. Finally, define the events
\begin{align*}
	L_1&:=\{l(x,0;\Omega_1)\vee l(x',0;\Omega_2)\vee l(x,0;\Omega_3)\vee l(x',0;\Omega_3)\leq |x-x'|/4\}, \\
	L_2&:=\{\xi^P_0(x;\Omega_1)=\xi^P_0(x';\Omega_2)=\xi^P_0(x;\Omega_3)=\xi^P_0(x';\Omega_3)=1\}\\
	&\qquad \cap\{T_{1,2}^{\text{sim}}\leq |x-x'|/4\}\cap\{T_{3,3}^{\text{sim}}\leq |x-x'|/4\}.
\end{align*}
Conditioned on these events the two random walks stay far enough apart. Note that the events $L_1$ and $L_2$ are disjoint. Since the probabilities of the complements of both sets have exponential bounds in $x'$ by Lemma A.1 in \cite{Birkner_et_al_2012} and Equation \eqref{eq:increment_tail_bound}, we know that $\mathbb P (L_1^c\cap L_2^c)$ has an exponential tail bound in $x'/4$. So,
\begin{align}\label{eq: split_variation_distance_into_disjoint_subsets}
	&\left |\mathbb P(\Xi_{x,x'}^\mathrm{joint}=w)-\mathbb P(\Xi_{x,x'}^\mathrm{ind}=\omega) \right | \nonumber\\
	&\qquad\leq \left |\mathbb P(\{\Xi_{x,x'}^\mathrm{joint}=w\}\cap L_1)-\mathbb P(\{\Xi_{x,x'}^\mathrm{ind}=\omega\}\cap L_1) \right | \nonumber\\
	&\qquad\qquad+\left |\mathbb P(\{\Xi_{x,x'}^\mathrm{joint}=w\}\cap L_2)-\mathbb P(\{\Xi_{x,x'}^\mathrm{ind}=\omega\}\cap L_2) \right | \nonumber \\
	&\qquad\qquad+Ce^{-cx'/4}.
\end{align}

\begin{figure}[h]
\centering
\includegraphics[width=0.9\textwidth]{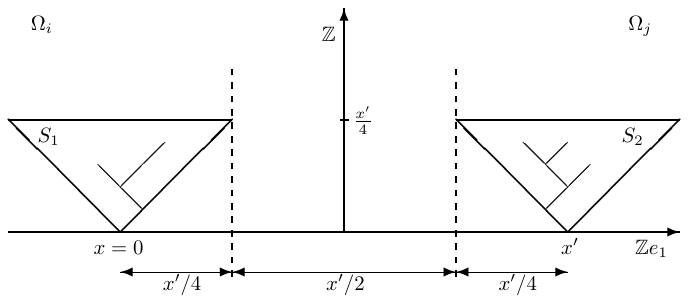}
\caption{Conditioned on either of the events $L_1$ and $L_2$ both regeneration increments depend on percolation cluster and permutations only in the sets $S_1$ and $S_2$ respectively. Since both sets have at least distance $x'/2$, we can use the mixing property of the environment to bound the difference of probabilities of a joint versus an independent pair of regeneration increments.}
\label{fig_total_variation_distance}
\end{figure}

On the event $L_1$ the regeneration increments are supported on the sets $S_1$ and $S_2$. Thus, we only have to use the space-mixing property of $K$ once to obtain immediately
\begin{align}\label{eq: total_variation_distance_L1}
	\left |\mathbb P(\{\Xi_{x,x'}^\text{joint}=w\}\cap L_1)-\mathbb P(\{\Xi_{x,x'}^\text{ind}=w\}\cap L_1) \right | \leq \alpha_{x'/2}.
\end{align}
Very similar to what we did in the proof of Lemma \ref{lem:increments_are_ergodic} we obtain the bounds on our second term by summing over all possible endpoints for the increment. For every site $(y,k)\in S_1$ there exist events 
\begin{align*}
	A_1^{(y,k)} (\Omega_i)&\in \sigma(\omega(z,l);\Omega_i), \tilde \omega(z,l;\Omega_i): (z,l)\in S_1)\quad \text{ and}\\ 
	A_2(\Omega_j) &\in \sigma(\omega(z,l;\Omega_j), \tilde \omega(z,l;\Omega_j) : |z-x'|\leq l),
\end{align*}
such that 
\begin{align*}
	&\{\Xi_{x,x'}^\text{joint}=w\}\cap\left \{\left (X_{T_{3,3}^\mathrm{sim}}(\Omega_3), T_{3,3}^\mathrm{sim}\right )=(y,k)\right \}\cap L_2\\
	&\qquad =A_1^{(y,k)}(\Omega_3)\cap\{\xi^P_k(y;\Omega_3)=1\}\cap A_2(\Omega_3)
\end{align*}
and 
\begin{align*}
	&\{\Xi_{x,x'}^\text{ind}=w\}\cap\left \{\left (X_{T_{1,2}^\mathrm{ind}}(\Omega_1), T_{1,2}^\mathrm{ind}\right )=(y,k)\right \}\cap L_2\\
	&\qquad =A_1^{(y,k)}(\Omega_1)\cap\{\xi^P_k(y;\Omega_1)=1\}\cap A_2(\Omega_2).
\end{align*}
Therefore 
\begin{align}\label{eq: total_variation_distance_L2}
	&\left |\mathbb P(\{\Xi_{x,x'}^\text{joint}=w\}\cap L_2)-\mathbb P(\{\Xi_{x,x'}^\text{ind}=w\}\cap L_2) \right | \nonumber \\
	&\qquad \leq \sum_{(y,k)\in S_1} \left |\mathbb P\left (A_1^{(y,k)}(\Omega_3)\cap\{\xi^P_k(y;\Omega_3)=1\}\cap A_2(\Omega_3)
\right )\right .  \nonumber \\
	&\qquad \qquad \qquad  - \left . \mathbb P\left (A_1^{(y,k)}(\Omega_1)\cap\{\xi^P_k(y;\Omega_1)=1\}\right ) \mathbb P (A_2(\Omega_2)) \right |  \nonumber \\
	&\qquad \leq \sum_{(y,k)\in S_1} \phi_{x'/4} + \alpha^X_{x'/4} \leq |S_1|\left ( \phi_{x'/4} + \alpha^X_{x'/4} \right ). 
\end{align}
Using $|S_1|=x'^2/8$, we can combine the three previous estimates in Equations \eqref{eq: split_variation_distance_into_disjoint_subsets}, \eqref{eq: total_variation_distance_L1} and \eqref{eq: total_variation_distance_L2}  to obtain
\begin{align*}
	& \left \|{\mathbb P}^{\mathrm{joint}}_{x,x'} (\Xi_1=\cdot)-{\mathbb P}^{\mathrm{ind}}_{x,x'} (\Xi_1=\cdot)\right \|_{\text{TV}}\\
	&\qquad =\sup_{w \in\mathcal W \cup \{\Delta\}}\left |\mathbb P(\Xi_{x,x'}^\mathrm{joint}=w)-\mathbb P(\Xi_{x,x'}^\mathrm{ind}=w) \right | \\
	&\qquad \leq Ce^{-cx'/4}+\alpha_{x'/2} +\frac{x'^2}{8}\left ( \phi_{x'/4} + \alpha^X_{x'/4} \right ).
\end{align*}
The conclusion of the lemma follows since all mixing coefficients are exponentially decreasing. 
\end{proof}

We have established that the total variation distance between the laws of two independent walks and two walks on the same cluster becomes small if the walks start far apart. Now, we need estimates on the probabilities to find two independent walks closer together or further apart after some time. For this, we compare it with standard Brownian motion and estimate exit probabilities from an annulus.

\begin{lemma}[Escape time from an annulus, cf.~Lemma 3.6 in \cite{Birkner_et_al_2012}]\label{lem:escape_time}
Let $U$ be the linear, bijective map that decomposes the inverse covariance matrix from Theorem \ref{thm:aCLT}, $\Sigma^{-1}=U^T U$. Write for $r>0$
\begin{align*}
	h(r)&:=\inf\{k\in\mathbb N: ||U(\hat X_k-\hat X'_k)||_\infty\leq r\}\\
	H(r)&:=\inf\{k\in\mathbb N: ||U(\hat X_k-\hat X'_k)||_\infty\geq r\}\\
\end{align*}
and set for $r_1<r<r_2$
\begin{align*}
	f_d(r;r_1,r_2)=\begin{cases}
		\frac{\log(r)-\log(r_1)}{\log(r_2)-\log(r_1)}\quad & \text{for }d\geq 3 \\
		\frac{r_1^{2-d}-r^{2-d}}{r_1^{2-d}-r_2^{2-d}}\quad & \text{for }d=2.
	\end{cases}
\end{align*}
Then for every $\epsilon>0$ there are large constants $R$ and $\tilde R$ such that for all $r_2>r_1>R$ and $r_2-r_1>\tilde R$ and for all starting points $x,y\in\mathbb Z^d$ such that $r=||U(x-y)||_\infty$, $r_1<r<r_2$, 
\begin{align*}
	(1-\epsilon)f_d(r;r_1,r_2)\leq \tilde{\mathbb P}_{x,y}^\mathrm{ind}(H(r_2)<h(r_1))\leq (1+\epsilon)f_d(r;r_1,r_2).
\end{align*}	
\end{lemma}

\begin{proof}
Under the law $\tilde{\mathbb P}^\mathrm{ind}$ the two copies of the random walk $(\hat X_k)_{k\in\mathbb N}$ and $(\hat X'_k)_{k\in\mathbb N}$ are independent and their difference is again a random walk with finite variance and zero mean. By Theorem 2.2 in \cite{Merlevede_2003} or Theorem 4.3 in \cite{Rio_2013} we get a functional central limit theorem for $(\hat X_k-\hat X'_k)_{k\in \mathbb N}$ under the same assumptions as in Theorem \ref{thm:aCLT}. The limit law is Brownian motion with some covariance operator $\Sigma$. Since the covariance matrix $\Sigma$ is symmetric and positive semi-definite it has a Cholesky decomposition. It has full rank and so the inverse has a decomposition $\Sigma^{-1}=U^T U$, where $U$ has full rank as well. Then the limit law of the random walk $(X_n)$ under the map $U$ has identity covariance matrix. If we define $h$ and $H$ as above, we can compare the random walks to the standard estimates of the exit probability from annuli for Brownian motion, which gives the conclusion.
\end{proof}

\begin{lemma}[Separation lemma, cf.~Lemma 3.8 in \cite{Birkner_et_al_2012}]\label{lem:separation_lemma}
For dimension $d\geq 2$ there are constants $b_1,b_2\in(0,1/2)$, $b_3>0$, $b_4\in(0,1)$, $C>0$ such that for large $n$
\begin{align}\label{eq: exit_prob_for_joint}
	\tilde{\mathbb P}_{0,0}^{\mathrm{joint}}\left (H(n^{b_1})\geq n^{b_2}\right )\leq Ce^{-b_3n^{b_4/2}}.
\end{align} 
\end{lemma}
\begin{proof}
For the proof, we have to be a little bit more careful as our environment $\xi^K$ does not have the Markov property and the regeneration increments are not independent. In the first step of the proof of Equation \eqref{eq: exit_prob_for_joint} we observe that instead of forcing two paths on the same cluster to move in opposite directions by choosing $\tilde \omega$, we can as well choose the percolation $\omega$ to our needs. In this case, we get the required bounds, Equation (3.29) in \cite{Birkner_et_al_2012}, with the further advantage that the construction depends on the percolation only. This allows us to rely on the Markov property of $\xi^P$. Furthermore, define the event
\begin{align*}
	A_n:=\{(\hat X_n, \hat X'_n) \text{ has reached distance } n^{b_1} \text{ in at most } n^{3b_1}+n^{b_6}\text{ steps}\}.
\end{align*}
Following the proof in \cite{Birkner_et_al_2012} we know from their Equations (3.34) and (3.35) that there exist $b_1\in(0,1/6)$ and $b_6\in(0,1/2)$ such that $\tilde{\mathbb P}_{x,y}^{\mathrm{joint}}(A_n)>\delta$ for some $\delta>0$ and uniformly in $x,y$. This bound is based on the escape time estimates of Lemma \ref{lem:escape_time}. We can pick $b_2\in(3b_1\vee b_6, 1/2)$ such that $n^{b_2}\geq n^{3b_1}+n^{b_6}$. We get the required upper bound for the probability to fail  to reach the distance $n^{b_1}$ at least $n^{b_4}$  times by looking only at every second regeneration increment and then using mixing properties to bound dependencies between them. This way we get
\begin{align*}
	\tilde{\mathbb P}_{0,0}^{\mathrm{joint}}\left ( \bigcap_{k=1}^{n^{b_4}}A_n^c\right ) &\leq \tilde{\mathbb P}_{0,0}^{\mathrm{joint}}\left ( \bigcap_{k=1, k \text{ odd}}^{n^{b_4}}A_n^c\right ) \\
	&\leq \tilde{\mathbb P}_{0,0}^{\mathrm{joint}}(A_1^c) \tilde{\mathbb P}_{0,0}^{\mathrm{joint}}\left ( \bigcap_{k=3, k \text{ odd}}^{n^{b_4}}A_n^c\right ) +\alpha_{n^{3b_1}+n^{b_6}}\\
	&\leq \ldots \leq \prod_{k=1, k \text{ odd}}^{n^{b_4}} \tilde{\mathbb P}_{0,0}^{\mathrm{joint}}\left (A_k^c\right )+\frac{n^{b_4}}{2}\alpha_{n^{3b_1}+n^{b_6}}\\
	&\leq(1-\delta)^{n^{b_4/2}}+C n^{b_4} e^{-cn^{3b_1}-cn^{b_6}}\\
	&\leq C n^{b_4}e^{-b_3n^{b_4/2}} \\
	&\leq Ce^{-b_3n^{b_4/2}}
\end{align*}
for $n$ large enough, $b_3:=\min\{-\log(1-\delta), c\}$ and $b_4<\min\{6b_1\vee 2b_6, b_2-3b_1\vee b_6\}$. By construction these attempts take at most 
\begin{align*}
	n^{b_4}(n^{3b_1}+n^{b_6})\leq n^{b_2-3b_1\vee b_6}n^{3b_1\vee b_6}=n^{b_2}
\end{align*}
steps. This proves Equation \eqref{eq: exit_prob_for_joint} for $d\geq 3$ and similarly in $d=2$, see \cite{Birkner_et_al_2012}.
\end{proof}

\begin{lemma}[cf.~Lemma 3.10 in \cite{Birkner_et_al_2012}]\label{lem:Birkner_3_10}
	For $d\geq 2$ there exist constants $b,C>0$ such that for every pair of bounded Lipschitz functions $f,g:\mathbb R^d\rightarrow \mathbb R$ with Lipschitz constants $L_f$ and $L_g$ respectively
\begin{align*}
	&\left | \tilde{\mathbb E}_{0,0}^\mathrm{joint}\left [ f\left ( \frac{\tilde X_n-n\tilde\mu}{\sqrt n}\right )g\left ( \frac{\tilde X'_n-n\tilde\mu}{\sqrt n}\right )\right ]-\tilde{\mathbb E}_{0,0}^\mathrm{ind}\left [ f\left ( \frac{\tilde X_n-n\tilde\mu}{\sqrt n}\right )g\left ( \frac{\tilde X'_n-n\tilde\mu}{\sqrt n}\right )\right ]\right | \\
	&\quad\leq   C\left ( 1+||f||_\infty +L_f\right )\left ( 1+||g||_\infty +L_g\right )n^{-b},
\end{align*}
where $\tilde\mu:=\mathbb E[\tau_1]\vec\mu$ and $\vec \mu$ is as in Theorem \ref{thm:aCLT}.
\end{lemma}
\begin{proof}
The proof remains almost the same as in \cite{Birkner_et_al_2012}, using the separation lemma, Lemma \ref{lem:separation_lemma}, and a coupling of dependent and independent $\Xi$-chains, introduced in their Lemma 3.9. Furthermore, we do not have standard large deviation estimates. Instead, we need use the Markov inequality together with an estimate on the expectation of the product of mixing random variables, e.g.~Theorem 17.2.2 in \cite{Ibragimov_Linnik_1971} to get
\begin{align*}
	\tilde {\mathbb P}_{0,0}^{\mathrm{joint}}\left (T_n^{\mathrm{sim}}\geq Kn\right )&\leq e^{-Kn}\tilde{\mathbb E}_{0,0}^{\mathrm{joint}} \left [ \prod_{k=1}^n e^{\tau_k^{\mathrm{joint}}}\right ]\\
	&\leq e^{-Kn}\sum_{k=0}^n \alpha_{2m}^{\frac{D}{D+2}}\tilde{\mathbb E}_{0,0}^{\mathrm{joint}}\left [ e^{\tau_k^{\mathrm{joint}}}\right ]^k\\
	&\leq n \alpha_{2m}^{\frac{D}{D+2}} \exp\left (\log\tilde{\mathbb E}_{0,0}^{\mathrm{joint}}\left [ e^{\tau_k^{\mathrm{joint}}}\right ]n-Kn  \right ).
\end{align*}
Since the time increments $\tau_n^{\mathrm{sim}}$ have exponential moments under the joint law by Equation \eqref{eq:increment_tail_bound}, we can choose the constant $K>\log\tilde {\mathbb E}_{0,0}^{\mathrm{joint}}[ e^{\tau_k^{\mathrm{sim}}}]$ to get exponential tail bounds
\begin{align}
	\tilde {\mathbb P}_{0,0}^{\mathrm{joint}}(T_n^{\mathrm{sim}}\geq Kn)\leq Ce^{-cn}.
\end{align}
This shows that Equation (3.47) in \cite{Birkner_et_al_2012} holds in our case and completes the proof by following their steps.
\end{proof}

\begin{proof}[Proof of Theorem \ref{thm:qCLT}]
As in \cite{Birkner_et_al_2012} we show that for any bounded Lipschitz function $f:\mathbb R^d\rightarrow \mathbb R$
\begin{align}\label{eq_quenched_law_goal}
	\left |  E_\xi \left [f\left ( \frac{X_n - n\vec\mu}{\sqrt n} \right )\right ] -\Phi(f)\right | \xrightarrow{n\rightarrow \infty }0 \quad \text{for }\tilde{\mathbb P}\text{-a.e.~}\xi^K,
\end{align}
where $\Phi(f)=\int f(x)\Phi(\mathrm{d}x)$ and $\Phi$ is a non-trivial $d$-dimensional normal law and $\vec \mu$ is as in Theorem \ref{thm:aCLT}. We do this by finding an upper bound of different terms and show that each of these terms converge individually as $n\rightarrow \infty$. Let $L_f$ be the Lipschitz constant of $f$ and write
\begin{align}
	&\left |  E_\xi \left [f\left ( \frac{X_n-n\vec\mu}{\sqrt n} \right )\right ] -\Phi(f)\right | \nonumber \\
	&\quad\leq  \left|  E_\xi \left [f\left ( \frac{X_n-n\vec\mu}{\sqrt n} \right )\right ] - E_\xi \left [f\left ( \frac{\tilde X_{[n/\mathbb E \tau_1]}-n\vec\mu}{\sqrt {n/\mathbb E \tau_1}} \frac{1}{\sqrt{\mathbb E \tau_1}}\right )\right ]  \right | \label{eq_quenched_law_triangle_inequality_1st} \\
	&\qquad + \left |  E_\xi \left [f\left ( \frac{\tilde X_{[n/\mathbb E \tau_1]}-n\vec\mu}{\sqrt {n/\mathbb E \tau_1}} \frac{1}{\sqrt{\mathbb E \tau_1}}\right )\right ] -\Phi(f)\right | \label{eq_quenched_law_triangle_inequality_2nd},
\end{align}
where we write $[n]$ for the integer part of an index $n$. In Term \eqref{eq_quenched_law_triangle_inequality_1st} we split the position of the random walk according to
\begin{align}
	\frac{X_n-n\vec\mu}{\sqrt n}={\frac{X_n-\tilde X_{V_n}}{\sqrt n}} + {\frac{\tilde X_{V_n}-\tilde X_{[n/\mathbb E[\tau_1]]}}{\sqrt n}} + 
	{\frac{\tilde X_{[n/\mathbb E[\tau_1]}-n\vec\mu}{\sqrt {n/\mathbb E[\tau_1]}}}\frac{1}{\sqrt{\mathbb E[\tau_1]}},
\end{align}
where
\begin{align*}
	V_n:=\max \{k>0: T_k\leq n\}.
\end{align*}
Conditioning on three suitable events (see terms (i)-(iii) in Equation \eqref{eq_quenched_law_triangle_inequality_1st_detail}) and using the properties of Lipschitz functions we get for constants $0< \gamma' < 1/2<\beta<1$ and $\gamma \in (\beta/2,1/2)$
\begin{align}\label{eq_quenched_law_triangle_inequality_1st_detail}	
	&\left|  E_\xi \left [f\left ( \frac{X_n-n\vec\mu}{\sqrt n} \right )\right ] - E_\xi \left [f\left ( \frac{\tilde X_{[n/\mathbb E \tau_1]}-n\vec\mu}{\sqrt {n/\mathbb E \tau_1}} \frac{1}{\sqrt{\mathbb E \tau_1}}\right )\right ]  \right |  \\
	&\leq \ L_f \frac{n^{\gamma'}+n^\gamma}{\sqrt n} + 2||f||_\infty \left (\underset{(i)}{P_\xi\left ( \left \|X_n-\tilde X_{V_n} \right \|\geq n^{\gamma'}\right )}+\underset{(ii)}{P_\xi\left (\left |V_n-\frac{n}{\mathbb E \tau_1}\right|\geq n^\beta\right )} \right )\nonumber \\
	&\qquad + 2||f||_\infty   \underset{(iii)}{P_\xi\left (\sup_{|k-n/\mathbb E \tau_1|<n^\beta} \left \|\tilde X_k-\tilde X_{[n/\mathbb E \tau_1]}\right \|\geq n^{\gamma}\right )}.  \nonumber
\end{align}
For term (i) in Equation \eqref{eq_quenched_law_triangle_inequality_1st_detail} we need to look at polynomial tails instead of logarithmic tails as in the original paper. For every $\epsilon>0$ we get from the Markov inequality and stationarity of the joint regeneration increments under the annealed law that
\begin{align*}
	&\tilde{\mathbb P}\left ( P_\xi \left ( \tau_n\geq n^{\gamma'}\right )>\epsilon \right ) \leq \frac{1}{\epsilon} \tilde{\mathbb E}\left ( P_\xi \left ( \tau_n\geq n^{\gamma'}\right ) \right ) = \frac{1}{\epsilon}\tilde{\mathbb P} \left ( \tau_n\geq n^{\gamma'}\right ) \leq \frac{C}{\epsilon}e^{-cn^{\gamma'}}
\end{align*}
is summable in $n$ for every $\epsilon>0$. 
We conclude by Borel-Cantelli that
\begin{align}
P_\xi\left ( \left \|X_n-\tilde X_{V_n}\right \| \geq n^{\gamma'}\right )  &\leq P_\xi\left ( \tau_{n} \geq n^{\gamma'}\right ) \xrightarrow{n\rightarrow \infty } 0 \text{ for }\tilde {\mathbb P}-\text{a.e.~} \xi^K,
\end{align}
since $||X_n-\tilde X_{V_n}|| \leq \tau_n$. For term (ii) in Equation \eqref{eq_quenched_law_triangle_inequality_1st_detail} we proceed as in \cite{Birkner_et_al_2012} and use their equations (3.64) and (3.67) to show almost sure convergence. Here we only need to remark that their Equation (3.64) holds by a law of the iterated logarithm for stationary mixing sequences, e.g.~Theorem 6.4 in \cite{Rio_2013}.
For term (iii) in Equation \eqref{eq_quenched_law_triangle_inequality_1st_detail} we use Equations (3.68) and (3.69) from \cite{Birkner_et_al_2012}, where Equation (3.69) holds by Inequality (I.6) in \cite{Rio_2013}.
Therefore, Term \eqref{eq_quenched_law_triangle_inequality_1st} converges for $\tilde{\mathbb P}$-a.e.~$\xi^K$ to $0$ as $n\rightarrow \infty$.

For Term \eqref{eq_quenched_law_triangle_inequality_2nd}, we choose $\Phi$ to be a rescaled normal law $\Phi(f(\cdot)):=\tilde \Phi ( f(\cdot / \sqrt{\mathbb E \tau_1}))$ and the almost sure convergence follows from Lemma \ref{lem:Birkner_3_10} together with Lemma 3.12 in \cite{Birkner_et_al_2012}. Lemma \ref{lem:Birkner_3_10} is used to control the covariance of two walks under $\mathbb P^{\mathrm{joint}}$ by comparing it to the variance of a single walker under the annealed law. Then Lemma 3.12 turns this into a quenched CLT for $\tilde X_n$. The proof of their Lemma 3.12 holds true in our case as there is a moderate deviation principle for stationary, strongly mixing sequences, Theorem 4 in \cite{Merlevede_Peligrad_Rio_2009}. We complete the proof of Theorem \ref{thm:qCLT} with the remark that any continuous bounded function can be approximated by bounded Lipschitz functions in a locally uniform way and Equation \eqref{eq_quenched_law_goal} holds for any continuous bounded function.
\end{proof}

\section{Open Questions and Further Work}
Many questions in this project remain open and can hopefully be answered with general progress in the field.  At this point we do not know, whether the percolation cluster only allows for easier proofs or whether it creates Brownian scaling limits in a broader class of weights $K$ compared to the full lattice. Therefore, one of the most important question to answer is, whether our aCLT can be extended to the full lattice $p=1$. While we failed to show non-degeneracy of the central limit theorem without additional assumptions on the moments of $K$, we are not aware of a counterexample either. 

We furthermore would like to establish whether our bound on the mixing coefficients in Theorem \ref{thm:aCLT} is sharp. So far we are not able to construct a suitable example since the percolation cluster destroys any trap.

Finally, the bound on the mixing coefficients in Theorem \ref{thm:qCLT} most certainly is not sharp and should probably depend on the dimension $d$. We would like to improve the bounds to polynomially mixing coefficients.

\section*{Acknowledgements}
I thank an unknown referee for his detailed review, which helped greatly to improve the paper. I would like to thank Noam Berger for his very simple and beautiful examples. Also, I am grateful to Renato Soares dos Santos and Florian Völlering for their remarks and helpful discussion. I am especially grateful to Nina Gantert for many critical questions and encouragement, and to the whole group for support and cake.


\end{document}